\documentclass[oneside,english]{amsart}
\usepackage[T1]{fontenc}
\usepackage[latin9]{inputenc}
\usepackage{amsthm}
\usepackage{amssymb}
\usepackage[all]{xy}

\makeatletter
\numberwithin{equation}{section}
\numberwithin{figure}{section}
\theoremstyle{plain}
\newtheorem{thm}{\protect\theoremname}[section]
  \theoremstyle{remark}
  \newtheorem{rem}[thm]{\protect\remarkname}
  \theoremstyle{definition}
  \newtheorem{defn}[thm]{\protect\definitionname}
  \theoremstyle{plain}
  \newtheorem{prop}[thm]{\protect\propositionname}
  \theoremstyle{plain}
  \newtheorem{lem}[thm]{\protect\lemmaname}
  \theoremstyle{plain}
  \newtheorem{cor}[thm]{\protect\corollaryname}

\makeatother

\usepackage{babel}
  \providecommand{\corollaryname}{Corollary}
  \providecommand{\definitionname}{Definition}
  \providecommand{\lemmaname}{Lemma}
  \providecommand{\propositionname}{Proposition}
  \providecommand{\remarkname}{Remark}
\providecommand{\theoremname}{Theorem}

\usepackage{cite}

\begin{document}

\title{Completion of continuity spaces with uniformly vanishing asymmetry}

\author{Alveen Chand, Ittay Weiss}
\begin{abstract}
The classical Cauchy completion of a metric space (by means of Cauchy
sequences) as well as the completion of a uniform space (by means
of Cauchy filters) are well-known to rely on the symmetry of the metric
space or uniform space in question. For qausi-metric spaces and quasi-uniform
spaces various non-equivalent completions exist, often defined on
a certain subcategory of spaces that satisfy a key property required
for the particular completion to exist. The classical filter completion
of a uniform space can be adapted to yield a filter completion of
a metric space. We show that this completion by filters generalizes
to continuity spaces that satisfy a form of symmetry which we call
uniformly vanishing asymmetry. 
\end{abstract}
\maketitle

\section{Introduction}

The theories of the completion of metric spaces and the completion
of uniform spaces are well-known and understood. There is little to
no doubt as to what completion should mean in these cases and there
are several (equivalent of course) constructions of the completions.
The situation is different when considering quasi-metric spaces and
quasi-uniform spaces. The lack of symmetry (see \cite{SH} for a detailed
account of symmetry and completions in the context of quantaloid enriched
categories) sabotages the standard completion constructions that work
in the symmetric case and the theory bifurcates with several different
notions of complete objects and different completion processes existing
in the literature (see e.g., \cite{Alemany96,Carlson71,Doitchinov88,Doitchinov91,Flagg99,Kivuvu08,Kivuvu09,Kunzi02,Lowen99,Render95,Romaguera02,Sherwood66}).

Consider the category ${\bf QMet}$ of quasi-metric spaces and uniformly
continuous functions, and the category ${\bf QUnif}$ of quasi-uniform
spaces and their morphisms. With a given quasi-metric space $(X,d)$
one can associate two quasi-uniform structures, one generated by the
entourages $U_{\varepsilon}=\{(x,y)\in X\times X\mid d(x,y)<\varepsilon\}$,
the other generated by the entourages $U^{\varepsilon}=\{(x,y)\in X\times X\mid d(y,x)<\varepsilon\}$,
giving rise to the two parallel functors in the diagram 
\[
\xymatrix{{\bf Met^{uva}}\ar[r] & {\bf QMet}\ar@<2pt>[r]^{U(-)}\ar@<-2pt>[r]_{U(-^{op})} & {\bf QUnif}}
\]
and it is natural to consider their equalizer. From the universal
property of the equalizer it follows that ${\bf Met^{uva}}$ extends
${\bf Met}$, the full subcategory of ${\bf QMet}$ spanned by the
ordinary metric spaces, but it is strictly larger. 

A quasi-metric space is the same thing as a $V$-\emph{continuity space
}or a $V$-\emph{space, }a concept introduced by Flagg in \cite{Flagg97a},
where $V$ is the value quantale $[0,\infty]$, viewed as a complete
lattice, with ordinary addition. Everything above can be repeated
with ${\bf Met}$ and ${\bf QMet}$ replaced, respectively, by ${\bf Met_{V}}$
(the category of symmetric $V$-spaces and uniformly continuous functions)
and ${\bf QMet_{V}}$ (the category of $V$-spaces and uniformly continuous
functions) for any value quantale $V$. In more detail, the aim of
this work is the construction of the dotted functors in the commutative
diagram
\[
\xymatrix{ & {\bf QMet_{V}\ar@<2pt>[r]\ar@<-2pt>[r]} & {\bf QUnif}\\
{\bf Met_{V}}\ar@{^{(}->}[r]\ar@{..>}[d] & {\bf Met_{V}^{uva}}\ar[u]^{e}\ar[r]\ar@{..>}[d] & {\bf Unif}\ar[u]\ar[d]\\
{\bf cMet_{V}}\ar@{^{(}->}[r] & {\bf cMet_{V}^{uva}}\ar[r] & {\bf cUnif}
}
\]
where the three lower vertical arrows are completion functors, thus
showing that the classical completion extends to the equalizer. In
more detail, in the diagram, the lower right vertical functor is the
standard construction of the completion of a uniform space via minimal
Cauchy filters. We show that this construction extends to separated
$V$-spaces in the equalizer. The construction is a metric re-incarnation
of that giving rise to the lower right functor. Most of the existing
notions of completions of $V$-spaces, when restricted to symmetric
spaces, yield (essentially) the same completion but these constructions
bifurcate to non-isometric completions for general $V$-spaces. It
is quite straightforward to manually check for most completions in
the literature that when restricted to $\mathbf{Met_{V}^{uva}}$,
the results are isometric. We thus expand the domain of the definition
of the standard completion to what appears to be the maximum possible.

In the context of the completion of metric spaces, one of the striking differences between the completion by means of Cauchy sequences and Cauchy filters is that the former requires a quotient construction, identifying sequences of distance $0$, while the latter enjoys a canonical choice of representative, namely the round filter $\mathcal F_\succ $ generated by fattening the elements in $\mathcal F$. In the construction of the completion we give below we also treat round filters in the full generality of $V$-spaces and we exhibit the 'roundification' process as a left adjoint on appropriately constructed categories. 

The plan of the article is as follows. Section $2$ recounts some
basic facts on value quantales and $V$-spaces following \cite{Flagg97a}.
Section $3$ introduces the concept of uniformly vanishing asymmetry,
the notion of symmetry required for the completion, which is then
presented in Section $4$.

\section{Value quantales and $V$-spaces}

Recall that a \emph{complete lattice} $L$ is a poset which possesses,
for all $S\subseteq L$, a meet $\bigwedge S$ and a join $\bigvee S$.
The top and bottom elements are denoted, respectively, by $\infty$
and $0$. The \emph{well-above }relation $\succ$ is derived from
the poset structure in $L$ (or any poset) as follows. For $a,b\in L$,
$a$ is said to be well-above $b$, denoted by $a\succ b$, or $b\prec a$,
if, given any $S\subseteq L$ such that $b\ge\bigwedge S$, there
exists $s\in S$ with $a\ge s$. 

A \emph{value quantale}, as introduced in \cite{Flagg97a} by Flagg,
is a pair $(V,+)$ where $V$ is a complete lattice and $+$ is an
associative and commutative binary operation on $V$ such that
\begin{itemize}
\item $x+0=x$
\item $x=\bigwedge\{y\in V\mid y\succ x\}$
\item $x+\bigwedge S=\bigwedge(x+S)$ 
\item $a\wedge b\succ0$
\end{itemize}
for all $x\in V$, $S\subseteq V$, and $a,b\in V$ with $a\succ0$
and $b\succ0$ ($x+S$ means $\{x+s\mid s\in S\}$). When the ambient
value quantale $V$ is clear from the context, we will write $\varepsilon\succ0$
as shorthand for the claim that $\varepsilon\in V$ and that $\varepsilon\succ0$
holds in $V$.
\begin{rem}
More commonly, a quantale is defined by the duals of the axioms above,
but in the context of this work we adhere to Flagg's original notation. 
\end{rem}

\subsection{Value quantale fundamentals}

We list those properties of value quantales that are needed for the
proofs that follow. We provide no arguments for the claims we make
in this section since the proofs are either immediate or are found
in \cite{Flagg97a}. Let $V$ be a value quantale. 
\begin{itemize}
\item For all $x,y,z\in V$,

\begin{enumerate}
\item if $x\succ y$, then $x\ge y$
\item if $x\succ y$ and $y\ge z$, then $x\succ z$
\item if $x\ge y$ and $y\succ z$, then $x\succ z$.
\end{enumerate}
\item For all $x,y,a,b\in V$, if $x\le y$ and $a\le b$, then $x+a\le y+b$. 
\item For all $\varepsilon\succ0$, there exists $\delta\succ0$ such that
$\delta+\delta\le\varepsilon$. More generally, for all $a\in V$ and
$n\ge1$ there exists $\delta\succ0$ such that $n\cdot\delta\le\varepsilon$,
where $n\cdot\delta$ denotes the $n$-fold addition of $\delta$
with itself. 
\item For all $x,z\in V$, if $x\prec z$, then there exists $y\in V$ with
$x\prec y\prec z$ (this result is known as the \emph{interpolation
property}).
\item Fix $b\in V$. Since $b+\square:V\to V$ preserves meets, it has a
left adjoint denoted by $\square-b:V\to V$, characterized by the
property that $a-b\le c\iff a\le b+c$, for all $a,c\in V$. Among
the numerous properties of this notion of subtraction, the one we
will use is $(a-b)-c=a-(b+c)$. 
\item For all $a\in V$, we have $a=\bigwedge\{a+\varepsilon\mid\varepsilon\succ0\}$.
Consequently, for all $a,b\in V$, $a\le b$ if, and only if, $a\le b+\varepsilon$
for all $\varepsilon\succ0$. 
\end{itemize}
Value quantales are the objects of a $2$-category $\mathbb{V}$ as
follows. A morphism $\alpha:V\to W$ in $\mathbb{V}$ is a monotone
function of the underlying lattices such that
\begin{itemize}
\item $\alpha(0)=0$, and
\item $\alpha(a+b)\le\alpha(a)+\alpha(b)$
\end{itemize}
for all $a,b\in V$. Each hom set $\mathbb{V}(V,W)$ is given a poset
structure by declaring, for $\alpha,\beta:V\to W$, that $a\le\beta$
precisely when $\beta(a)\le\alpha(a)$ for all $a\in V$. Interpreting
the poset $\mathbb{V}(V,W)$ as a category thus describes the $2$-cells
in the $2$-category $\mathbb{V}$.

\subsection{$V$-spaces\label{sub:catPers}}

Flagg introduced value quantales to replace the traditional non-negative
extended real numbers and act as generalized codomains for distance
functions $d:X\times X\to V$. Thus, a \emph{$V$-space} (called a
$V$-\emph{continuity space} in \cite{Flagg97a}), is a triple $(X,d,V)$
where $V$ is a value quantale, $X$ is a set, and $d:X\times X\to V$
is a function satisfying
\begin{itemize}
\item $d(x,x)=0$; and
\item $d(x,z)\le d(x,y)+d(y,z)$ 
\end{itemize}
for all $x,y,z\in X$. A $V$-space $(X,d,V)$ is \emph{symmetric}
if $d(x,y)=d(y,x)$ for all $x,y\in X$ and it is called \emph{separated
}if, for all $x,y\in X$, the equalities $d(x,y)=0=d(y,x)$ imply
$x=y$. 

For a given value quantale $V$, the category ${\bf Met_{V}}$ consists
of all symmetric $V$-spaces as objects and all \emph{uniformly continuous
}mappings as morphisms $f:X\to Y$ (i.e., those functions satisfying
that for all $\varepsilon\succ0$ there exists $\delta\succ0$ such that
$d(f(x_{1}),f(x_{2}))\le\varepsilon$ whenever $d(x_{1},x_{2})\le\delta$).
Ignoring size issues, the assignment $V\mapsto{\bf Met_{V}}$ extends
to a $2$-functor ${\bf Met_{-}}:\mathbb{V}\to{\bf Cat}$ into the
$ $$2$-category of categories. In more detail, if $\alpha:V\to W$
is a morphism of value quantales, then ${\bf Met_{\alpha}}:{\bf Met_{V}}\to{\bf Met_{W}}$
sends a $V$-space $(X,d)$ to $(X,\alpha_{*}d)$ where $\alpha_{*}d(x,y)=\alpha(d(x,y))$,
which is easily seen to be a $W$-space. If now $\beta:V\to W$ is
another morphism such that $\alpha\le\beta$, then it is immediately
verified that there is a natural transformation ${\bf Met_{\alpha}\to{\bf Met}_{\beta}}$,
where the component at $X$ is the identity function on the underlying
sets. 

Similarly, one defines the categories ${\bf QMet_{V}}$ of $V$-spaces
and again one obtains a similar $2$-functor ${\bf QMet}_{-}:\mathbb{V}\to{\bf Cat}$.
The full subcategory ${\bf QMet_{V}^{0}}$ of ${\bf QMet_{V}}$ spanned
by the separated $V$-spaces is reflective, with the reflector $-_{0}:{\bf QMet_{V}}\to{\bf QMet_{V}^{0}}$
mapping $X$ to $X_{0}=X/\sim$, where $x\sim y$ precisely when $d(x,y)=d(y,x)=0$.
Similarly, the full subcategory ${\bf Met_{V}^{0}}$ of ${\bf Met_{V}}$
spanned by the separated symmetric $V$-spaces is reflective. It is
immediate that if $(X,d)$ is a $V$-space, then so is the \emph{dual
space} $X^{op}=(X,d^{op})$, where $d^{op}(x,y)=d(y,x)$. 
\begin{rem}
$V$-spaces are in fact enriched $V$-categories. However, notice
that then enriched functors correspond to non-expanding functions
rather than the uniformly continuous ones we consider.
\end{rem}
$V$-spaces are general enough to capture all topological spaces in
the sense that for every topological space $X$, there is a value
quantale $V$ such that $X$ is $V$-metrizable (Theorem $4.15$ in
\cite{Flagg97a}). Further, the category ${\bf Top}$ is equivalent
to the category ${\bf QMet}_{T}$ whose objects are all pairs $(V,X)$
where $V$ is a value quantale and $X$ is a $V$-space, and a morphism
$(V,X)\to(W,Y)$ is a continuous function $f:X\to Y$ (see \cite{Weiss13}
for more details).

\section{Uniformly vanishing asymmetry}

We introduce now a class of spaces with a sufficient amount of symmetry
to allow for the classical completion via Cauchy filters to carry
through. 

For a $V$-space $X$, a point $x\in X$, and $\varepsilon\succ0$
let ${\bf B}_{\varepsilon}(x)=\{y\in X\mid d(x,y)\le\varepsilon\}$ and
similarly let ${\bf B}^{\varepsilon}(x)=\{y\in X\mid d(y,x)\le\varepsilon\}$.
We extend the notation ${\bf B}_{\varepsilon}(x)$ to subsets $S\subseteq X$
by defining ${\bf B}_{\varepsilon}(S)=\bigcup_{s\in S}{\bf B}_{\varepsilon}(s)$,
with ${\bf B}^{\varepsilon}(S)$ defined similarly. Notice that ${\bf B}_{\varepsilon}(x)$
in $X^{op}$ is precisely ${\bf B}^{\varepsilon}(x)$ in $X$. The set
${\bf B}_{\varepsilon}(x)$ is a closed set in the topology generated
by the sets of the form $\{y\in X\mid d(x,y)\prec\varepsilon\}$,
where $x$ varies over $X$ and $\varepsilon\succ0$ varies in $V$
(see Theorem $4.4$ in \cite{Flagg97a}). That topology is denoted
by $\mathcal{O}(X)$ and one obtains the functor $\mathcal{O}:{\bf QMet_{V}}\to{\bf Top}$.
A straightforward verification shows that a $V$-space $X$ gives
rise to a quasi-uniform space $U(X)$, where the entourages are generated
by $\{(x,y)\in X\times X\mid d(x,y)\le\varepsilon\}$, where $\varepsilon\succ0$
varies in $V$, giving rise to a fully faithful functor $U:{\bf QMet_{V}}\to{\bf QUnif}$. 

For a $V$-space $X$, the conditions 
\begin{itemize}
\item for all $x\in X$ and for all $\varepsilon\succ0$ there exists $\delta\succ0$
such that ${\bf B}^{\delta}(x)\subseteq{\bf B}_{\varepsilon}(x)$ and
such that ${\bf B}_{\delta}(x)\subseteq{\bf B}^{\varepsilon}(x)$ (any
such $\delta$ will be called a \emph{modulus of symmetry }for $\varepsilon$);
\item the identity function $X^{op}\to X$ is a homeomorphism;
\item $\mathcal{O}(X)=\mathcal{O}(X^{op})$
\end{itemize}
are equivalent. If $X$ satisfies these conditions, then $X$ is said
to have\emph{ vanishing asymmetry.} Similarly, the conditions
\begin{itemize}
\item for all $\varepsilon\succ0$ there exists $\delta\succ0$ such that
if $d(y,x)\le\delta$, then $d(x,y)\le\varepsilon$ (any such $\delta$
will be called a \emph{uniform modulus of symmetry }for $\varepsilon$);
\item the identity functions $X^{op}\to X$ and $X\to X^{op}$ are uniformly
continuous;
\item $U(X)=U(X^{op})$
\end{itemize}
are equivalent. If $X$ satisfies these conditions, then $X$ is said
to have \emph{uniformly vanishing asymmetry}. Clearly, if $X$ has
uniformly vanishing asymmetry, then $X$ has vanishing asymmetry. 

Let ${\bf Met_{V}^{va}}$ and ${\bf Met_{V}^{uva}}$ be the full subcategories
of ${\bf QMet_{V}}$ spanned by the spaces with vanishing asymmetry
and the spaces with uniformly vanishing asymmetry, respectively. Consider
the diagram
\[
\xymatrix{{\bf Met_{V}^{uva}}\ar@{^{(}->}[r]\ar@{^{(}->}[d] & {\bf QMet_{V}}\ar@<2pt>[r]^{U(-)}\ar@<-2pt>[r]_{U(-^{op})}\ar[d]_{id} & {\bf QUnif}\ar@<-2pt>[d]\ar@<2pt>[d]\\
{\bf Met_{V}^{va}}\ar@{^{(}->}[r] & {\bf QMet_{V}}\ar@<2pt>[r]^{\mathcal{O}(-)}\ar@<-2pt>[r]_{\mathcal{O}(-^{op})} & {\bf Top}
}
\]
where the square on the left consists of inclusions, and the right
vertical arrows are the standard constructions of the topology associated
to a quasi-uniform space. The diagram commutes as long as one does
not incorrectly mix different functors in the square on the right,
and we note that from the definition of (uniformly) vanishing asymmetry,
the top and bottom parts of the diagram are equalizers.

\section{Completion}

From this point onwards, we fix a value quantale $V$ and a $V$-space
$X$. We develop the relevant ingredients for constructing a completion
of $X$ as the set of all minimal Cauchy filters on $X$. 

Recall that a \emph{filter }on a set $X$ is a non-empty collection
$\mathcal{F}\subseteq\mathcal{P}(X)$ such that $A\in\mathcal{F}\implies B\in\mathcal{F}$
for all $A\subseteq B\subseteq X$, and $A\cap B\in\mathcal{F}$ for
all $A,B\in\mathcal{F}$ (we do not require that $\emptyset\notin\mathcal{F}$,
so in particular the power-set $\mathcal{P}(X)$ is a filter, the
unique filter containing the empty set, referred to as an \emph{improper
filter}). A \emph{filter base} is a collection $\mathcal{B}\subseteq\mathcal{P}(X)$
such that for all $A,B\in\mathcal{B}$ there exists $C\in\mathcal{B}$
with $C\subseteq A\cap B$. It follows immediately that a filter base
$\mathcal{B}$ gives rise to a filter $\mathcal{F}$, the least filter
containing $\mathcal{B}$, given explicitly by $\mathcal{F}=\{D\subseteq X\mid\exists C\in\mathcal{B},C\subseteq D\}$.
By a filter (resp. filter base) on a $V$-space is meant a filter
(resp. filter base) on the underlying set. 

A filter $\mathcal{F}$ is said to \emph{converge }to $x\in X$, written
$\mathcal{F}\to x$, if ${\bf B}_{\varepsilon}(x)\in\mathcal{F}$ for
all $\varepsilon\succ0$. Convergence interpreted in $X^{op}$ is
referred to as op-convergence, thus $\mathcal{F}$ \emph{op-converges}
to $x$, denoted by $\mathcal{F}\to^{op}x$, when ${\bf B}^{\varepsilon}(x)\in\mathcal{F}$
for all $\varepsilon\succ0$.
\begin{defn}
A filter $\mathcal{F}$ on $X$ is said to be a \emph{Cauchy filter
}if for all $\varepsilon\succ0$ there exists $x\in X$ such that
${\bf B}_{\varepsilon}(x)\in\mathcal{F}$. If, moreover, $\mathcal{F}$
does not contain any proper Cauchy subfilter, then $\mathcal{F}$
is called a \emph{minimal Cauchy filter}. 
\end{defn}
$X$ is said to be \emph{Cauchy complete }if every proper Cauchy filter
on $X$ converges. The dual notion of a Cauchy filter is that of an
\emph{op-Cauchy filter}, namely when for all $\varepsilon\succ0$
there exists an $x\in X$ with ${\bf B}^{\varepsilon}(x)\in\mathcal{F}$,
that is $\mathcal{F}$ is op-Cauchy in $X$ precisely when $\mathcal{F}$
is Cauchy in $X^{op}$. The $V$-space $X$ is \emph{op-Cauchy complete
}if every proper op-Cauchy filter on $X$ op-converges. For spaces
with uniformly vanishing asymmetry introduced above, the dual concepts
of Cauchy completeness and op-Cauchy completeness coincide. A \emph{completion
}of $X$ is a Cauchy complete $V$-space $\hat{X}$ together with
an isometry $X\to\hat{X}$ with dense image in $\hat{X}$. 
\begin{defn}
A filter $\mathcal{F}$ in $X$ is said to be a \emph{round filter
}if for all $F\in\mathcal{F}$ there exists $\varepsilon\succ0$ such
that ${\bf B}_{\varepsilon}(x)\in\mathcal{F}$ implies ${\bf B}_{\varepsilon}(x)\subseteq F$,
for all $x\in X$.
\end{defn}
The omitted proof of the following result is completely formal. 
\begin{prop}
\label{prop:Cauchy plus round implies minimal cauchy}If $\mathcal{F}$
is Cauchy and round, then $\mathcal{F}$ is minimal Cauchy. 
\end{prop}
Let $\mathcal{F}$ be a filter. Consider the collection $\{{\bf B}_{\varepsilon}(F)\mid F\in\mathcal{F},\,\,\varepsilon\succ0\}$,
which is a filter base since ${\bf B}_{\varepsilon\wedge\delta}(F\cap F')\subseteq{\bf B}_{\varepsilon}(F)\cap{\bf B}_{\delta}(F')$
(recalling that $\varepsilon\wedge\delta\succ0$). The generated filter
is denoted by $\mathcal{F}_{\succ}$. 
\begin{prop}
\label{prop:Cauchy implies cauchy}If $\mathcal{F}$ is Cauchy, then
$\mathcal{F}_{\succ}$ is Cauchy.\end{prop}
\begin{proof}
Let $\varepsilon\succ0$ and $\delta\succ0$ with $2\cdot\delta\le\varepsilon$.
Let $x\in X$ with ${\bf B}_{\delta}(x)\in\mathcal{F}$, and so ${\bf B}_{\delta}({\bf B}_{\delta}(x))\in\mathcal{F}_{\succ}$.
Then ${\bf B}_{\varepsilon}(x)\in\mathcal{F}_{\succ}$ follows by ${\bf B}_{\delta}({\bf B}_{\delta}(x))\subseteq{\bf B}_{2\cdot\delta}(x)\subseteq{\bf B}_{\varepsilon}(x)$. \end{proof}
\begin{lem}
\label{lem:UBA implies round}If $X$ has uniformly vanishing asymmetry
and $\mathcal{F}\ne\mathcal{P}(X)$ is a filter on $X$, then $\mathcal{F}_{\succ}$
is round. \end{lem}
\begin{proof}
It suffices to show that for a given basis element ${\bf B}_{\varepsilon}(F)\in\mathcal{F}_{\succ}$
there exists a $\delta\succ0$ such that if ${\bf B}_{\delta}(y)\in\mathcal{F}_{\succ}$,
then ${\bf B}_{\delta}(y)\subseteq{\bf B}_{\varepsilon}(F)$. Let $\delta_{1}\succ0$
with $2\cdot\delta_{1}\le\varepsilon$ and let $\delta_{2}\succ0$
be a uniform modulus of symmetry for $\delta_{1}$. Set $\delta=\delta_{1}\wedge\delta_{2}$.
Suppose that ${\bf B}_{\delta}(y)\in\mathcal{F}_{\succ}$ for some
$y\in X$. Then ${\bf B}_{\delta}(y)\supseteq{\bf B}_{\varepsilon'}(F')\supseteq F'$
for some $F'\in\mathcal{F}$ and $\varepsilon'\succ0$. Choose some $x\in F\cap F'$.
Then, $x\in{\bf B}_{\delta}(y)$ and so $d(y,x)\le\delta$. To show
now that ${\bf B}_{\delta}(y)\subseteq{\bf B}_{\varepsilon}(F)$, notice
that if $z\in{\bf B}_{\delta}(y)$, then $d(y,z)\le\delta$, and so
$d(x,z)\le d(x,y)+d(y,z)\le\delta_{1}+\delta_{1}\le\varepsilon$,
and thus $z\in{\bf B}_{\varepsilon}(F)$. \end{proof}
\begin{cor}
\label{cor:problematicSolved}If $X$ has uniformly vanishing asymmetry
and $\mathcal{F}\ne\mathcal{P}(X)$ is a Cauchy filter on $X$, then
$\mathcal{F}_{\succ}$ is a minimal Cauchy filter.
\end{cor}

\begin{cor}
If $X$ has uniformly vanishing asymmetry, then a filter $\mathcal{F}$
is minimal Cauchy if, and only if, $\mathcal{F}$ is Cauchy and round. \end{cor}
\begin{proof}
One direction is Proposition~\ref{prop:Cauchy plus round implies minimal cauchy}.
For the other direction, if $\mathcal{F}$ is minimal Cauchy, then
$\mathcal{F}=\mathcal{F}_{\succ}$, which is round. 
\end{proof}
The results above translate to interesting categorical relations between
Cauchy and round filters, as we now show. Let ${\bf Fil_{V}}$ be
the category whose objects are all pairs $(X,\mathcal{F})$ where
$X$ is a $V$-space with uniformly vanishing asymmetry and $\mathcal{F}$
is a filter on $X$. The morphisms $f:(X,\mathcal{F})\to(Y,\mathcal{G})$
are uniformly continuous functions $f:X\to Y$ with the property that
$f(\mathcal{F})\supseteq\mathcal{G}$, where $f(\mathcal{F})=\{S\subseteq Y\mid f^{-1}(S)\in\mathcal{F}\}$
(which is easily seen to be a filter). Let ${\bf RFil_{V}}$ and ${\bf CFil_{V}}$
be the full subcategories of ${\bf Fil_{V}}$ spanned by round filters
and by Cauchy filters, respectively. Let ${\bf CRFil_{V}}={\bf CFil_{V}}\cap{\bf RFil_{V}}$.
Finally, let ${\bf rFil_{V}}$ be the full subcategory of ${\bf Fil_{V}}$
spanned by the \emph{restricted} objects, i.e., objects $(X,\mathcal{F})$
where $\mathcal{F}\ne\mathcal{P}(X)$ if $X\ne\emptyset$. Similarly
one defines the other restricted full subcategories. Consider the
diagram 
\[
\xymatrix{ &  & {\bf rFil_{V}}\ar[d]\ar@<-4pt>[ddll]_{(-)_{\succ}}\\
 &  & {\bf Fil_{V}}\ar[d]\\
{\bf rRFil_{V}}\ar[uurr]\ar[r]^{=} & {\bf RFil_{V}}\ar@<-2pt>@{..>}[ur]\ar[r] & {\bf Met_{V}^{uva}}\ar@{-->}@<-4pt>[u]\ar@{..>}@<4pt>[u]\ar@{..>}@<4pt>[l]\ar@{-->}@<4pt>[r] & {\bf CFil_{V}\ar@{-->}@<-2pt>[ul]\ar[l]} & {\bf rCFil_{V}\ar[uull]\ar[l]}\ar@<-2pt>[ddll]_{\quad(-)_{\succ}}\\
 &  & {\bf CRFil_{V}}\ar@<-2pt>[ur]\ar[ul]\ar[u]\\
 &  & {\bf rCRFil_{V}}\ar[u]^{=}\ar@<-2pt>[uurr]\ar[uull]
}
\]
where the upwards directed arrows in both diamonds are inclusion functors
and all of the arrows in the smaller diamond pointing towards the
centre are the obvious forgetful functors. The other arrows (which
are detailed below), with the exception of the pair on the upper left
side of the outer diamond, are all adjunctions, with the left adjoint
depicted on top or to the left of its right adjoint. The left adjoint
${\bf Met_{V}^{uva}\to{\bf Fil_{V}}}$ sends $X$ to $(X,\mathcal{P}(X))$
while the right adjoint sends $X$ to $(X,\{X\})$. For these functors
the dotted and the dashed triangles commute. We note that in the degenerate
case $V=\{0=\infty\}$, the inner diamond reduces to identity functors,
${\bf Met_{V}^{uva}\cong{\bf Set}}$, and ${\bf Fil_{V}}$ is the
category of filters introduced in \cite{Blass77}. 
\begin{rem}
Regarding the forgetful functor $p:{\bf CRFil_{V}\to Met_{V}^{uva}},$
recall that the fiber over an object $X$ is the category consisting
of all of the objects in ${\bf CRFil_{V}}$ that project to $X$ and
all morphisms that project to the identity on $X$. This category
is essentially a set and is precisely the completion of $X$ we construct
below. \end{rem}
\begin{prop}
\label{prop:Problemaitc}The construction $(X,\mathcal{F})\mapsto(X,\mathcal{F}_{\succ})$
is the object part of a functor $(-)_{\succ}:{\bf Fil_{V}\to{\bf Fil_{V}}}$
which further sends $f:\mathcal{F}\to\mathcal{G}$ to $f:\mathcal{F}_{\succ}\to\mathcal{G}_{\succ}$.
The restriction of this functor to ${\bf rFil_{V}}$ gives rise to
the functor at the top left of the diagram above. \end{prop}
\begin{proof}
Note that uniform continuity of $f:X\to Y$ implies that for all $\varepsilon\succ0$
there exists $\delta\succ0$ such that $f^{-1}({\bf B}_{\varepsilon}(S))\supseteq{\bf B}_{\delta}(f^{-1}(S))$,
for all $S\subseteq Y$. Now, to show that $(-)_{\succ}$ is functorial,
suppose that $f:(X,\mathcal{F})\to(Y,\mathcal{G})$ is a morphism,
i.e., that $f(\mathcal{F})\supseteq\mathcal{G}$, and we need to show
that $f:(X,\mathcal{F}_{\succ})\to(Y,\mathcal{G}_{\succ})$ is a morphism,
i.e., that $f(\mathcal{F}_{\succ})\supseteq\mathcal{G}_{\succ}$.
Indeed, if $G\in\mathcal{G_{\succ}}$, then $G\supseteq{\bf B}_{\varepsilon}(G')$
for some $G'\in\mathcal{G}$ and $\varepsilon\succ0$. It thus follows
that $f^{-1}(G)\supseteq f^{-1}({\bf B}_{\varepsilon}(G'))\supseteq{\bf B}_{\delta}(f^{-1}(G'))$
for a suitable $\delta\succ0$. Since $f^{-1}(G')\in\mathcal{F}$
we conclude that $f^{-1}(G)\in\mathcal{F}_{\succ}$. The claim about
the image of the functor is Corollary~\ref{cor:problematicSolved}. 
\end{proof}
Note that generally speaking $\mathcal{G}\supseteq\mathcal{G}_{\succ}$
but strict inclusion may hold even if $\mathcal{G}$ is already round.
The fact that for a round filter that is also Cauchy $\mathcal{G}_{\succ}=\mathcal{G}$
(by minimality) is crucial in the following proof.
\begin{prop}
The functor $(-)_{\succ}:{\bf rFil_{V}}\to{\bf rRFil_{V}}$ restricts
to a functor $(-)_{\succ}:{\bf rCFil_{V}\to rCRFil_{V}}$. This functor
is left adjoint to the inclusion functor ${\bf rCRFil_{V}\to{\bf rCFil_{V}}}.$\end{prop}
\begin{proof}
The claim about the restriction landing in Cauchy filters is Proposition~\ref{prop:Cauchy implies cauchy}.
To establish that $(-)_{\succ}$ is left adjoint to the inclusion,
we need to show for a Cauchy and round filter $\mathcal{G}$ on $Y$
and an arbitrary Cauchy filter $\mathcal{F}$ on $X$, that $f:(X,\mathcal{F}_{\succ})\to(Y,\mathcal{G})$
is a morphism, i.e., that $f(\mathcal{F}_{\succ})\supseteq\mathcal{G}$,
if, and only if, $f:(X,\mathcal{F})\to(Y,\mathcal{G})$ is a morphism,
i.e., $f(\mathcal{F})\supseteq\mathcal{G}$. Since $\mathcal{F}\supseteq\mathcal{F}_{\succ}$,
it follows that $f(\mathcal{F})\supseteq f(\mathcal{F}_{\succ})$,
and thus one of the implications is trivial. Assume now that $f(\mathcal{F})\supseteq\mathcal{G}$,
and we need to show that $f(\mathcal{F}_{\succ})\supseteq\mathcal{G}$.
Let $G\in\mathcal{G}$. As $\mathcal{G}$ is Cauchy and round, thus
minimal Cauchy, we have that $\mathcal{G}_{\succ}=\mathcal{G}$, and
so there exists some $G'\in\mathcal{G}$ and $\varepsilon\succ0$
with $G\supseteq{\bf B}_{\varepsilon}(G')$. Then $f^{-1}(G)\supseteq f^{-1}({\bf B}_{\varepsilon}(G'))\supseteq{\bf B}_{\delta}(f^{-1}(G'))$
for some $\delta\succ0$, and since $f^{-1}(G')\in\mathcal{F}$, we
conclude that $f^{-1}(G)\in\mathcal{F}_{\succ}$. 
\end{proof}
This concludes the description of the functors in the diagram above.
We now turn to the details of the completion construction. For $x\in X$
let $\mathcal{F}_{x}$ be the filter generated by the filter base
$\mathcal{B}_{x}=\{{\bf B}_{\varepsilon}(x)\mid\varepsilon\succ0\}$,
which is clearly Cauchy. The dual construction is the filter $\mathcal{F}^{x}$
generated by the filter base $\mathcal{B}^{x}=\{{\bf B}^{\varepsilon}(x)\mid\varepsilon\succ0\}$.
For general $V$-spaces, a filter may be Cauchy without being op-Cauchy
and $\mathcal{F}_{x}=\mathcal{F}^{x}$ need not hold. 
\begin{prop}
\label{prop:compiffopcomp}If $X$ has vanishing asymmetry, then $\mathcal{F}_{x}=\mathcal{F}^{x}$
for all $x\in X$. If $X$ has uniformly vanishing asymmetry, then
a filter $\mathcal{F}$ is Cauchy if, and only if, it is op-Cauchy.
Consequently, $X$ is Cauchy complete if, and only if, it is op-Cauchy
complete.\end{prop}
\begin{proof}
To show that $\mathcal{F}_{x}=\mathcal{F}^{x}$ it suffices to argue
on basis elements. If $X$ has vanishing asymmetry, then given ${\bf B}_{\varepsilon}(x)\in\mathcal{B}_{x}$
let $\delta\succ0$ be such that ${\bf B}^{\delta}(x)\subseteq{\bf B}_{\varepsilon}(x)$,
which thus shows that ${\bf B}_{\varepsilon}(x)\in\mathcal{F}^{x}$ ,
and so $\mathcal{F}_{x}\subseteq\mathcal{F}^{x}$. The reverse inequality
follows similarly. Suppose now that $X$ has uniformly vanishing asymmetry
and that $\mathcal{F}$ is Cauchy. Given $\varepsilon\succ0$, let
$\delta\succ0$ be a corresponding modulus of uniform symmetry. There
is then $x\in X$ with ${\bf B}_{\delta}(x)\in\mathcal{F}$, and since
${\bf B}_{\delta}(x)\subseteq{\bf B}^{\varepsilon}(x)$, it follows that
${\bf B}^{\varepsilon}(x)\in\mathcal{F}$, and so $\mathcal{F}$ is op-Cauchy.
The reverse implication is similar. The last assertion in the proposition
follows since $\mathcal{F}\to x$ is equivalent to $\mathcal{F}_{x}\subseteq\mathcal{F}$,
and $\mathcal{F}\to^{op}x$ is equivalent to $\mathcal{F}^{x}\subseteq\mathcal{F}$. \end{proof}
\begin{prop}
If $X$ has uniformly vanishing asymmetry, then $\mathcal{F}_{x}$
is round. \end{prop}
\begin{proof}
Given ${\bf B}_{\varepsilon}(x)\in\mathcal{F}_{x}$, let $\delta_{1}\succ0$
with $2\cdot\delta_{1}\le\varepsilon$ and let $\delta_{2}\succ0$
be a uniform modulus of symmetry for $\delta_{1}$. Set $\delta=\delta_{1}\wedge\delta_{2}$
and suppose ${\bf B}_{\delta}(y)\in\mathcal{F}_{x}$. Then clearly
$x\in{\bf B}_{\delta}(y)$, thus $d(y,x)\le\delta$, implying that
$d(x,y)\le\delta_{1}$. Now, to show that ${\bf B}_{\delta}(y)\subseteq{\bf B}_{\varepsilon}(x)$,
notice that if $z\in{\bf B}_{\delta}(y)$, then $d(x,z)\le d(x,y)+d(y,z)\le\delta_{1}+\delta_{1}\le\varepsilon$,
and so $z\in{\bf B}_{\varepsilon}(x)$. 
\end{proof}
For subsets $S,T\subseteq X$, let $d(S,T)=\bigwedge_{s\in S,t\in T}d(s,t)$,
and for collections $\mathcal{E}_{1},\mathcal{E}_{2}\subseteq\mathcal{P}(X)$,
let $d(\mathcal{E}_{1},\mathcal{E}_{2})=\bigvee_{S\in\mathcal{E}_{1},T\in\mathcal{E}_{2}}d(S,T)$,
giving rise to a function $d:\mathcal{P}(\mathcal{P}(X))\times\mathcal{P}(\mathcal{P}(X))\to V$.
Let $\tilde{X}\subseteq\mathcal{P}(\mathcal{P}(X))$ be the set of
all proper (i.e., $\mathcal{P}(X)$ is excluded) Cauchy filters on
$X$ and let $\hat{X}\subseteq\tilde{X}$ be the set of all minimal
Cauchy filters. It is easy to see that if $\mathcal{B}_{\mathcal{F}}$
and $\mathcal{B}_{\mathcal{G}}$ are filter bases for $\mathcal{F}$
and $\mathcal{G}$ respectively, then $d(\mathcal{F},\mathcal{G})=d(\mathcal{B}_{\mathcal{F}},\mathcal{B}_{\mathcal{G}})$.
(Alternatively, notice that $\mathcal{F}_{x}$ is $(\mathcal{P}_{x})_{\succ}$,
where $\mathcal{P}_{x}$ is the principal filter on $x$.)

The following computation is convenient to record for the proofs below. 
\begin{prop}
\label{prop:conv}Suppose that $S\subseteq{\bf B}^{\delta}(x)$ and
$T\subseteq{\bf B}_{\varepsilon}(y)$ and $S, T \ne\emptyset$.
Then $d(S,T)\le d(x,y)+\delta+\varepsilon$ and $d(y,x)\le d(T,S)+\delta+\varepsilon$.\end{prop}
\begin{proof}
Let $s\in S$ and $t\in T$ be arbitrary. Then $d(S,T)\le d(s,t)\le d(s,x)+d(x,y)+d(y,t)\le\delta+d(x,y)+\varepsilon$,
which is the first inequality. By the distributivity law in $V$,
the second inequality will follow by showing that $d(y,x)\le d(t,s)+\delta+\varepsilon$
for all $s\in S$ and $t\in T$. Indeed, $d(y,x)\le d(y,t)+d(t,s)+d(s,x)\le\varepsilon+d(t,s)+\delta$. \end{proof}
\begin{lem}
If $X$ has uniformly vanishing asymmetry, then $(\tilde{X},d)$ is
a $V$-space, which itself has uniformly vanishing asymmetry. \end{lem}
\begin{proof}
$d(\mathcal{F},\mathcal{F})=0$ since for all $F,F'\in\mathcal{F}$,
$F\cap F'\ne\emptyset$. To establish that $d(\mathcal{F},\mathcal{H})\le d(\mathcal{F},\mathcal{G})+d(\mathcal{G},\mathcal{H})$
it suffices to show, for fixed $\varepsilon\succ0$, $F\in\mathcal{F}$,
and $H\in\mathcal{H}$, that there exists $G\in\mathcal{G}$ such
that $d(F,H)\le d(F,G)+d(G,H)+\varepsilon$. Let $\delta\succ0$ be
such that $2\cdot\delta\le\varepsilon$ and let $\delta'\succ0$ be
a uniform modulus of symmetry for $\delta$, and set $\eta=\delta\wedge\delta'$.
As $\mathcal{G}$ is Cauchy, there is $x\in X$ such that $G={\bf B}_{\eta}(x)\in\mathcal{G}$.
And then 
\begin{eqnarray*}
d(F,G)+d(G,H)+\varepsilon & \ge & \bigwedge_{f\in F,y,z\in G,h\in H}d(f,y)+d(z,h)+2\cdot\delta\\
 & \ge & \bigwedge_{f\in F,y,z\in G,h\in H}d(f,y)+d(y,x)+d(x,z)+d(z,h)\\
 & \ge & \bigwedge_{f\in F,h\in H}d(f,h)=d(F,H)
\end{eqnarray*}
as required for showing that $\tilde{X}$ is a $V$-space. 

To show that $\tilde{X}$ has uniformly vanishing asymmetry, let $\varepsilon\succ0$
be given and let $\eta\succ0$ with $2\cdot\eta\le\varepsilon$. Let
$\delta_{1}\succ0$ be a uniform modulus of symmetry for $\eta$, and
$\delta\succ0$ with $2\cdot\delta\le\delta_{1}$. Suppose that $d(\mathcal{G},\mathcal{F})\le\delta$,
which means that $d(G,F)\le\delta$ for all $F\in\mathcal{F}$ and
$G\in\mathcal{G}$. To show that $d(\mathcal{F},\mathcal{G})\le\varepsilon$
it suffices to show that $d(F_{0},G_{0})\le\varepsilon$ for fixed
$F_{0}\in\mathcal{F}$ and $G_{0}\in\mathcal{G}$. Since $\mathcal{F}$
is Cauchy (and thus op-Cauchy) and since $\mathcal{G}$ is Cauchy,
there exist $x,y\in X$ with ${\bf B}^{\delta'}(x)\in\mathcal{F}$
and ${\bf B}_{\delta'}(y)\in\mathcal{G}$, where $\delta'\succ0$
satisfies $2\cdot\delta'\le\delta$. Let $S=F_{0}\cap{\bf B}^{\delta'}(x)$
and $T=G_{0}\cap{\bf B}_{\delta'}(y)$. Then, using Proposition~\ref{prop:conv}
(here and in the following computation), $d(y,x)\le d(T,S)+\delta\le2\cdot\delta\le\delta_{1}$
and thus $d(x,y)\le\eta$. Finally, $d(F_{0,}G_{0})\le d(S,T)\le d(x,y)+\eta\le2\cdot\eta\le\varepsilon$,
as required.
\end{proof}
For any $V$-space, setting $x\sim y$ whenever $d(x,y)=d(y,x)=0$
is an equivalence relation, and $X_{0}$, the set of equivalence classes
becomes a separated $V$-space where the distance function is given
by $d([x],[y])=d(x,y)$. We note that if $X$ has vanishing asymmetry,
then $d(x,y)=0$ implies $d(y,x)=0$ and if $X$ is also separated,
then $\mathcal{O}(X)$ is Hausdorff. In particular, the following
result (whose proof is immediate and thus omitted), implies that if
$X$ has vanishing asymmetry, then $X_{0}$ is Hausdorff.
\begin{prop}
If $X$ has (uniformly) vanishing asymmetry, then so does $X_{0}$.\end{prop}
\begin{thm}
If $X$ has uniformly vanishing asymmetry, then $\hat{X}$ is isometric
to $\tilde{X}_{0}$.\end{thm}
\begin{proof}
For any two Cauchy filters $\mathcal{F},\mathcal{G}$ on $X$, their
intersection is again a filter but it need not be Cauchy. However,
if $d(\mathcal{F},\mathcal{G})=0$, then $\mathcal{F}\cap\mathcal{G}$
is Cauchy. Indeed, let $\varepsilon\succ0$ and let $\delta_{1}\succ0$
with $2\cdot\delta_{1}\le\varepsilon$. Let $\delta_{2}\succ0$ be
a uniform modulus of symmetry for $\delta_{1}$, and let $\delta_{3}\succ0$
satisfy $2\cdot\delta_{3}\le\delta_{2}$. Set $\delta=\delta_{2}\wedge\delta_{3}$.
There exists $x\in X$ with ${\bf B}_{\delta}(x)\in\mathcal{F}$ and
$y\in X$ with ${\bf B}^{\delta}(y)\in\mathcal{G}$, and since $d(\mathcal{F},\mathcal{G})=0$
it follows that $d({\bf B}_{\delta}(x),{\bf B}^{\delta}(y))=0$, and
thus that $d(x,y)\le2\cdot\delta\le\delta_{1}$. Now, $d(y,s)\le\delta_{1}$
for all $s\in{\bf B}^{\delta}(y)$ and so $d(x,s)\le d(x,y)+d(y,s)\le2\cdot\delta_{1}\le\varepsilon$,
leading to ${\bf B}^{\delta}(y)\subseteq{\bf B}_{\varepsilon}(x)$. It
thus follows that ${\bf B}_{\varepsilon}(x)\in\mathcal{G}$, which establishes
that $\mathcal{F}\cap\mathcal{G}$ is Cauchy. 

It now follows that each equivalence class $[\mathcal{F}]$ contains
a unique minimal Cauchy representative. Indeed, it is easily seen
that $d(\mathcal{F},\mathcal{F}_{\succ})=0$ so that $\mathcal{F}_{\succ}\in[\mathcal{F}]$.
If $\mathcal{F}_{1},\mathcal{F}_{2}$ are two minimal Cauchy filters
with $\mathcal{F}_{1}\sim\mathcal{F}_{2}$, then $\mathcal{F}_{1}\cap\mathcal{F}_{2}$
is Cauchy so that minimality forces $\mathcal{F}_{1}=\mathcal{F}_{2}$.
The bijective isometry $\tilde{X}_{0}\to\hat{X}$ is thus given by
$[\mathcal{F}]\mapsto\mathcal{F}_{\succ}$. \end{proof}
\begin{cor}
If $X$ has uniformly vanishing asymmetry, then $\hat{X}$ is a separated
$V$-space with uniformly vanishing asymmetry. 
\end{cor}
Recall that when $X$ has uniformly vanishing asymmetry every the
filters $\mathcal{F}_{x}$ are round (and clearly Cauchy). We then
obtain the function $\iota:X\to\hat{X}$, given by $\iota(x)=\mathcal{F}_{x}$,
called the \emph{canonical embedding }(even though it is injective
if, and only if, $X$ is separated). 
\begin{lem}
\label{lem:canonicalEmbedIsIsom}If $X$ has uniformly vanishing asymmetry,
then the canonical embedding $\iota:X\to\hat{X}$ is an isometry. \end{lem}
\begin{proof}
Clearly, $d({\bf B}_{\varepsilon}(x),{\bf B}_{\delta}(y))\le d(x,y)$,
thus $d(\mathcal{B}_{x},\mathcal{B}_{y})\le d(x,y)$, and therefore
$d(\mathcal{F}_{x},\mathcal{F}_{y})\le d(x,y)$. For the other direction,
we will use the fact that $\mathcal{F}_{y}=\mathcal{F}^{y}$ (cf.
Proposition~\ref{prop:compiffopcomp}), so it suffices to show that
$d(x,y)\le d(\mathcal{B}_{x},\mathcal{B}^{y})$. To that end, let
$\rho\succ0$, and $\eta\succ0$ with $2\cdot\eta\le\rho$. Since in general $d(x,y)-\varepsilon-\delta\le d({\bf B}_{\varepsilon}(x),{\bf B}^{\delta}(y))$ we have $d(\mathcal{B}_{x},\mathcal{B}^{y})\ge d({\bf B}_{\eta}(x),{\bf B}^{\eta}(y))\ge(d(x,y)-\eta)-\eta=d(x,y)-(\eta+\eta)\ge d(x,y)-\rho$.
Thus, $d(x,y)\le d(\mathcal{B}_{x},\mathcal{B}^{y})+\rho$, and as
$\rho\succ0$ is arbitrary, the desired inequality follows. \end{proof}
\begin{cor}
If $X$ is separated and has uniformly vanishing asymmetry, then the
canonical embedding $\iota:X\to\hat{X}$ is injective. \end{cor}
\begin{lem}
\label{lem:CanonEmbIsDense}If $X$ has uniformly vanishing asymmetry,
then the image $\iota(X)$ is dense in $\hat{X}$. \end{lem}
\begin{proof}
Fix $\mathcal{G}\in\hat{X}$ and $\varepsilon\succ0$. Let $\delta\succ0$
be a uniform modulus of symmetry for $\varepsilon$, and since $\mathcal{G}$
is Cauchy we may find $x\in X$ with ${\bf B}_{\delta}(x)\in\mathcal{G}$.
To show that $d(\mathcal{G},\mathcal{F}_{x})\le\varepsilon$ it suffices
to show that $d(G,{\bf B}_{\eta}(x))\le\varepsilon$ for all $\eta\succ0$
and $G\in\mathcal{G}$. Let $y\in G\cap{\bf B}_{\delta}(x)$. Then
$d(x,y)\le\delta$ implies $d(G,{\bf B}_{\eta}(x))\le d(G\cap{\bf B}_{\delta}(x),x)\le d(y,x)\le\varepsilon$.\end{proof}
\begin{thm}
\label{thm:IsComplete}If $X$ has uniformly vanishing asymmetry,
then $\hat{X}$ is Cauchy complete. \end{thm}
\begin{proof}
It suffices to show that every proper Cauchy filter on $\hat{X}$ converges
to a minimal Cauchy filter on $X$. Let $\mathbb{A}$ be a Cauchy
filter on $\hat{X}$ and $\varepsilon\succ0$. Then there is a minimal
Cauchy filter $\mathcal{L}\in\hat{X}$ such that $\mathbf{B}_{\varepsilon}(\mathcal{L})=\{\mathcal{G}\in\hat{X}\mid d(\mathcal{L},\mathcal{G})\le\varepsilon\}$
is in $\mathbb{A}$. It is straightforward to verify that $\mathcal{F}=\{F\subseteq X\mid\exists A\in\mathbb{A},F\in\bigcap A\}$
is a filter. Next, to show that $\mathcal{F}$ is Cauchy, let $\varepsilon\succ0$
and $\delta_{1}\succ0$ with $4\cdot\delta_{1}\le\varepsilon$.
Further, let  $\delta_{2}\succ0$ be a uniform modulus of symmetry for $\delta_{1}$.
Note  $\delta _3 = \delta_{1}\wedge\delta_{2}\succ0$. Fix $\mathcal{G}\in\mathbf{B}_{\varepsilon}(\mathcal{L})$.
Since $\mathcal{G}$ is Cauchy, there is $y\in X$ such that $\mathbf{B}_{\delta_{3}}(y)\in\mathcal{G}$.
For $a\in\mathbf{B}_{\delta_{3}}(y)$, we have $d(x,a)\le d(x,y)+d(y,a)\le d(\mathbf{B}_{\delta_{1}}(x),\mathbf{B}_{\delta_{3}}(y))+2\cdot\delta_{1}+\delta_{3}\le4\cdot\delta_{1}\le\varepsilon$, thus $\mathbf{B}_{\delta_{3}}(y)\subseteq\mathbf{B}_{\varepsilon}(x)$
which implies $\mathbf{B}_{\varepsilon}(x)\in\mathcal{G}$. Since $\mathcal{G}\in\mathbf{B}_{\varepsilon}(\mathcal{L})$
is arbitrary, $\mathbf{B}_{\varepsilon}(x)\in\bigcap\mathbf{B}_{\varepsilon}(\mathcal{L})$,
thus $\mathbf{B}_{\varepsilon}(x)\in\mathcal{F}$.

Finally, we show that $\mathbb{A}$ converges to the minimal Cauchy
filter $\mathcal{F}_{\succ}$. Let $\varepsilon\succ0$ and $\delta_{1}\succ0$ 
with $2\cdot\delta_{1}\le\varepsilon$, and further let  $\delta_{2}\succ0$ be a 
uniform modulus of symmetry for $\delta_{1}$. 
Note that $\delta = \delta_{1}\wedge\delta_{2}\succ 0$. 
There is $\mathcal{L}\in\hat{X}$
such that $\mathbf{B}_{\delta}(\mathcal{L})\in\mathbb{A}$. Then it
suffices to show that $\mathbf{B}_{\delta}(\mathcal{L})\subseteq\mathbf{B}_{\varepsilon}(\mathcal{F}_{\succ})$.
Let $\mathcal{M}\in\mathbf{B}_{\delta}(\mathcal{L})$, $L_{0}\in\mathcal{L}$
and $F_{0}\in\mathcal{F}$. This means that there is $A\in\mathbb{A}$
such that $F_{0}\in\bigcap A$. Since $\mathbb{A}$ is a proper filter, $\mathbf{B}_{\delta}(\mathcal{L})\cap A\neq\emptyset$
and $d(\mathcal{L},\mathcal{G})\le\delta\le\delta_{2}$
implies $d(\mathcal{G},\mathcal{L})\le\delta_{1}$ for every $\mathcal{G}\in\mathbf{B}_{\delta}(\mathcal{L})\cap A$.
Since $F_{0}$ is also in $\mathcal{G}$, $d(F_{0},L_{0})\le\bigvee_{G\in\mathcal{G},L\in\mathcal{L}}d(G,L)=d(\mathcal{G},\mathcal{L})\le\delta_{1}$,
thus $d(F_{0},L_{0})\le\delta_{1}$. Since $F_{0}\in\mathcal{F}$
and $L_{0}\in\mathcal{L}$ are arbitrary, we obtain $d(\mathcal{F},\mathcal{L})\le\delta_{1}$.
Then $d(\mathcal{F}_{\succ},\mathcal{L})\le d(\mathcal{F}_{\succ},\mathcal{F})+d(\mathcal{F},\mathcal{L})\le0+\delta_{1}$
which implies that $d(\mathcal{F}_{\succ},\mathcal{M})\le d(\mathcal{F}_{\succ},\mathcal{L})+d(\mathcal{L},\mathcal{M})\le\delta_{1}+\delta\le\delta_{1}+\delta_{1}\le\varepsilon$,
thus $\mathcal{M}\in\mathbf{B}_{\varepsilon}(\mathcal{F}_{\succ})$.
Since $\mathcal{M}\in\mathbf{B}_{\delta}(\mathcal{L})$ is arbitrary,
$\mathbf{B}_{\delta}(\mathcal{L})\subseteq\mathbf{B}_{\varepsilon}(\mathcal{F}_{\succ})$
and since $\varepsilon\succ0$ is arbitrary, it follows that $\mathbb{A}$
converges to $\mathcal{F}_{\succ}$.
\end{proof}
Obviously, the construction $X\mapsto\hat{X}$ is functorial. The
following two corollaries follow by standard arguments from Lemma~\ref{lem:canonicalEmbedIsIsom},
Lemma~\ref{lem:CanonEmbIsDense}, and Theorem~\ref{thm:IsComplete}:
\begin{cor}
The universal property 
\[
\xymatrix{ & X\ar[dr]^{f}\ar[dl]_{\iota}\\
\hat{X}\ar[rr]_{F} &  & Y
}
\]
stating that for any Cauchy complete $V$-space $Y$ with uniformly
vanishing asymmetry and any uniformly continuous function $f$ there
exists a unique uniformly continuous extension $F$, holds for all
separated $V$-spaces $X$ with uniformly vanishing asymmetry. 
\end{cor}

\begin{cor}
Every separated $V$-space $X$ with uniformly vanishing asymmetry
has a completion, unique up to a unique isomorphism. 
\end{cor}
Relating back to the categorical point-of-view, i.e., to the $2$-functor
${\bf QMet_{-}}:\mathbb{V}\to{\bf Cat}$ from Section~\ref{sub:catPers},
the constructions above may be summarized as follows. Consider the
obvious 2-functors ${\bf sMet_{-}^{uva}},{\bf scMet_{-}^{uva}}:\mathbb{V}\to{\bf Cat}$
mapping $V$ to ${\bf sMet_{V}^{uva}}$ and to $s{\bf cMet_{V}^{uva}}$
(the categories of separated $V$-spaces with uniformly vanishing
asymmetry and of complete separated $V$-spaces with uniformly vanishing
asymmetry), respectively. The completion functor ${\bf sMet_{V}^{uva}}\to{\bf scMet_{V}^{uva}}$
defines a $2$-natural transformation $\hat{-}:{\bf sMet_{-}^{uva}}\to{\bf scMet_{-}^{uva}}$.

\bibliography{refs}{}

\begin{thebibliography}{10}

\bibitem{Alemany96}
E.~Alemany and S.~Romaguera.
\newblock On half-completion and bicompletion of quasi-metric spaces.
\newblock {\em Commentationes Mathematicae Universitatis Carolinae},
  37(4):749--756, 1996.

\bibitem{Blass77}
A.~Blass.
\newblock Two closed categories of filters.
\newblock {\em Fundamenta Mathematicae}, 94(2):129--143, 1977.

\bibitem{Carlson71}
J.~W. Carlson and T.~L. Hicks.
\newblock On completeness in quasi-uniform spaces.
\newblock {\em Journal of Mathematical Analysis and Applications}, 34:618--627,
  1971.

\bibitem{Doitchinov88}
D.~Doitchinov.
\newblock On completeness in quasi-metric spaces.
\newblock {\em Topology and its Applications}, 30:127--148, 1988.

\bibitem{Doitchinov91}
D.~Doitchinov.
\newblock A concept of completeness of quasi-uniform spaces.
\newblock {\em Topology and its Applications}, 38:205--217, 1991.

\bibitem{Flagg99}
B.~Flagg, R.~Kopperman, and P.~Sünderhauf.
\newblock Smyth completion as bicompletion.
\newblock {\em Topology and its Applications}, 91:169--180, 1999.

\bibitem{Flagg97a}
R.~C. Flagg.
\newblock Quantales and continuity spaces.
\newblock {\em Algebra Universalis}, 37:257--276, 1997.

\bibitem{SH}
H.~Heymans and I.~Stubbe.
\newblock Symmetry and {C}auchy completion of quantaloid-enriched categories.
\newblock {\em Theory Appl. Categ.}, 25:No. 11, 276--294, 2011.

\bibitem{Kivuvu08}
C.~M. Kivuvu and H.~P.~A. Kunzi.
\newblock A double completion for an arbitrary $\uppercase{T}_0$-quasi-metric
  space.
\newblock {\em The Journal of Logic and Algebraic Programming}, 76(2):251--269,
  2008.

\bibitem{Kivuvu09}
C.~M. Kivuvu and H.~P.~A. Kunzi.
\newblock The $\uppercase{B}$-completion of a $\uppercase{T}_0$-quasi-metric
  space.
\newblock {\em Topology and its Applications}, 156(12):2070--2081, 2009.

\bibitem{Kunzi02}
H.~P.~A. Kunzi and M.~P. Schellekens.
\newblock On the yoneda completion of a quasi-metric space.
\newblock {\em Theoretical Computer Science}, 278(1):159--194, 2002.

\bibitem{Lowen99}
R.~Lowen and D.~Vaughan.
\newblock A non quasi-metric completion for quasi-metric spaces.
\newblock 1999.

\bibitem{Render95}
H.~Render.
\newblock Nonstandard methods of completing quasi-uniform spaces.
\newblock {\em Topology and its Applications}, 62(2):101--125, 1995.

\bibitem{Romaguera02}
S.~Romaguera and M.~A. Sanchez-Granero.
\newblock Completions and compactifications of quasi-uniform spaces.
\newblock {\em Topology and its Applications}, 123(2):363--382, 2002.

\bibitem{Sherwood66}
H.~Sherwood.
\newblock On the completion of probabilistic metric spaces.
\newblock {\em Probability Theory and Related Fields}, 6(1):62--64, 1966.

\bibitem{Weiss13}
I.~Weiss.
\newblock A note on the metrizability of spaces.
\newblock {\em alg. universalis (to appear)}, 2013.

\end{thebibliography}
\bibliographystyle{plain}
\end{document}